\documentclass[11pt,english]{amsart}
\usepackage[T1]{fontenc}
\usepackage[latin1]{inputenc}
\usepackage{verbatim}
\usepackage{amstext}
\usepackage{amsthm}
\usepackage{amssymb}

\makeatletter

\providecommand{\tabularnewline}{\\}

\numberwithin{equation}{section}
\numberwithin{figure}{section}
\theoremstyle{definition}
\newtheorem*{defn*}{\protect\definitionname}
\theoremstyle{plain}
\newtheorem{thm}{\protect\theoremname}
\newenvironment{lyxlist}[1]
	{\begin{list}{}
		{\settowidth{\labelwidth}{#1}
		 \setlength{\leftmargin}{\labelwidth}
		 \addtolength{\leftmargin}{\labelsep}
		 }}
	{\end{list}}
\theoremstyle{plain}
\newtheorem{cor}[thm]{\protect\corollaryname}
\theoremstyle{definition}
\newtheorem{defn}[thm]{\protect\definitionname}
\theoremstyle{remark}
\newtheorem{rem}[thm]{\protect\remarkname}
\theoremstyle{plain}
\newtheorem{lem}[thm]{\protect\lemmaname}
\theoremstyle{plain}
\newtheorem{prop}[thm]{\protect\propositionname}

\usepackage{tikz}
\usepackage{pgfplots}

\makeatother

\makeatother

\usepackage{babel}
\providecommand{\corollaryname}{Corollary}
\providecommand{\definitionname}{Definition}
\providecommand{\lemmaname}{Lemma}
\providecommand{\propositionname}{Proposition}
\providecommand{\remarkname}{Remark}
\providecommand{\theoremname}{Theorem}

\begin{document}

\title[On the dynamics of the line operator $\Lambda_{\{2\},\{3\}}$]{On the dynamics of the line operator $\Lambda_{\{2\},\{3\}}$ on some arrangements of six lines}

\addtolength{\textwidth}{0mm}
\addtolength{\hoffset}{-0mm} 


\global\long\def\CC{\mathbb{C}}%
 
\global\long\def\BB{\mathbb{B}}%
 
\global\long\def\PP{\mathbb{P}}%
 
\global\long\def\QQ{\mathbb{Q}}%
 
\global\long\def\RR{\mathbb{R}}%
 
\global\long\def\FF{\mathbb{F}}%

\global\long\def\DD{\mathbb{D}}%
 
\global\long\def\NN{\mathbb{N}}%
\global\long\def\ZZ{\mathbb{Z}}%
 
\global\long\def\HH{\mathbb{H}}%
 
\global\long\def\Gal{{\rm Gal}}%

\global\long\def\bA{\mathbf{A}}%

\global\long\def\kP{\mathfrak{P}}%
 
\global\long\def\kQ{\mathfrak{q}}%
 
\global\long\def\ka{\mathfrak{a}}%
\global\long\def\kP{\mathfrak{p}}%
\global\long\def\kn{\mathfrak{n}}%
\global\long\def\km{\mathfrak{m}}%

\global\long\def\cA{\mathfrak{\mathcal{A}}}%
\global\long\def\cB{\mathfrak{\mathcal{B}}}%
\global\long\def\cC{\mathfrak{\mathcal{C}}}%
\global\long\def\cD{\mathcal{D}}%
\global\long\def\cH{\mathcal{H}}%
\global\long\def\cK{\mathcal{K}}%

\global\long\def\cF{\mathcal{F}}%
 
\global\long\def\cI{\mathfrak{\mathcal{I}}}%
\global\long\def\cJ{\mathcal{J}}%

\global\long\def\cL{\mathcal{L}}%
\global\long\def\cM{\mathcal{M}}%
\global\long\def\cN{\mathcal{N}}%
\global\long\def\cO{\mathcal{O}}%
\global\long\def\cP{\mathcal{P}}%
\global\long\def\cS{\mathcal{S}}%
\global\long\def\cW{\mathcal{W}}%

\global\long\def\cQ{\mathcal{Q}}%
\global\long\def\kBS{\mathfrak{B}_{6}}%

\global\long\def\a{\alpha}%
 
\global\long\def\b{\beta}%
 
\global\long\def\d{\delta}%
 
\global\long\def\D{\Delta}%
 
\global\long\def\L{\Lambda}%
 
\global\long\def\g{\gamma}%
\global\long\def\om{\omega}%

\global\long\def\G{\Gamma}%
 
\global\long\def\d{\delta}%
 
\global\long\def\D{\Delta}%
 
\global\long\def\e{\varepsilon}%
 
\global\long\def\k{\kappa}%
 
\global\long\def\l{\lambda}%
 
\global\long\def\m{\mu}%

\global\long\def\o{\omega}%
 
\global\long\def\p{\pi}%
 
\global\long\def\P{\Pi}%
 
\global\long\def\s{\sigma}%

\global\long\def\S{\Sigma}%
 
\global\long\def\t{\theta}%
 
\global\long\def\T{\Theta}%
 
\global\long\def\f{\varphi}%
 
\global\long\def\ze{\zeta}%

\global\long\def\deg{{\rm deg}}%
 
\global\long\def\det{{\rm det}}%

\global\long\def\Dem{Proof: }%
 
\global\long\def\ker{{\rm Ker}}%
 
\global\long\def\im{{\rm Im}}%
 
\global\long\def\rk{{\rm rk}}%
 
\global\long\def\car{{\rm car}}%
\global\long\def\fix{{\rm Fix( }}%

\global\long\def\card{{\rm Card }}%
 
\global\long\def\codim{{\rm codim}}%
 
\global\long\def\coker{{\rm Coker}}%

\global\long\def\pgcd{{\rm pgcd}}%
 
\global\long\def\ppcm{{\rm ppcm}}%
 
\global\long\def\la{\langle}%
 
\global\long\def\ra{\rangle}%

\global\long\def\Alb{{\rm Alb}}%
 
\global\long\def\Jac{{\rm Jac}}%
 
\global\long\def\Disc{{\rm Disc}}%
 
\global\long\def\Tr{{\rm Tr}}%
 
\global\long\def\Nr{{\rm Nr}}%

\global\long\def\NS{{\rm NS}}%
 
\global\long\def\Pic{{\rm Pic}}%

\global\long\def\Km{{\rm Km}}%
\global\long\def\rk{{\rm rk}}%
\global\long\def\Hom{{\rm Hom}}%
 
\global\long\def\End{{\rm End}}%
 
\global\long\def\aut{{\rm Aut}}%
 
\global\long\def\SSm{{\rm S}}%

\global\long\def\psl{{\rm PSL}}%
 
\global\long\def\cu{{\rm (-2)}}%
 
\global\long\def\mod{{\rm \,mod\,}}%
 
\global\long\def\cros{{\rm Cross}}%
 
\global\long\def\nt{z_{o}}%

\global\long\def\co{\mathfrak{\mathcal{C}}_{0}}%
\global\long\def\ci{\mathfrak{\mathcal{C}}_{1}}%
\global\long\def\ldt{\Lambda_{\{2\},\{3\}}}%
 
\author{Xavier Roulleau}
\begin{abstract}
The operator $\ldt$ acting on line arrangements is defined by associating
to a line arrangement $\cA$, the line arrangement which is the union
of the lines containing exactly three points among the double points
of $\cA$. We say that six lines not tangent to a conic form an \textit{unassuming}
arrangement if the singularities of their union are only double points,
but the dual line arrangement has six triple points, six $5$-points
and $27$ double points. The moduli space of unassuming arrangements
is the union of a point and a line. The image by the operator $\ldt$
of an unassuming arrangement is again an unassuming arrangement. We
study the dynamics of the operator $\ldt$ on these arrangements and
we obtain that the periodic arrangements are related to the Ceva arrangements
of lines. 
\end{abstract}

\maketitle

\section{Introduction}

In \cite{OSO}, we give the formal definitions of some operators acting
on line arrangements and configurations of points in the projective
plane. The aim of that paper was to show that these operators appear
quite naturally and ubiquitously when constructing or studying line
arrangements with remarkable properties (such as reflexion line arrangements,
simplicial arrangements, or Sylvester-Gallai line arrangements, i.e.,
line arrangements without double points), and therefore one should
start to study these operators for themselves, as well as for finding
interesting line arrangements. The present paper is about the operator
$\ldt$ and certain arrangements of six lines we define below. The
study of the operator $\ldt$ enabled us to find these simple looking
arrangements (see Figure \ref{fig:The-povera-arr}), which (to our
knowledge) have not been noticed before.

Recall that for $k\geq2$, we denote by $\cP_{\{k\}}$ the \textit{point
operator}, which to a line arrangement $\cA$, associates the set
$\cP_{\{k\}}(\cA)$ of $k$-points of $\cA$ i.e., the points where
exactly $k$ lines of $\cA$ meet. The order of $\cP_{\{k\}}(\cA)$
is denoted by $t_{k}(\cA)=|\cP_{\{k\}}(\cA)|$. For $k\geq2$, we
denote by $\cL_{\{k\}}$ (resp. $\cL_{k}$) the \textit{line operator}
which to a finite set of points $P$ associates the set $\cL_{\{k\}}(P)$
(resp. $\cL_{k}(P)$) of lines containing exactly (resp. at least)
$k$ points in $P$. The operator $\L_{\{m\},\{n\}}$ is defined by
$\L_{\{m\},\{n\}}=\cL_{\{n\}}\circ\cP_{\{m\}}$, and we also define
$\Psi_{\{m\},\{n\}}=\cP_{\{n\}}\circ\cL_{\{m\}}$. We denote by $\cD$
the dual operator, which to a set of points (respectively lines) associate
the corresponding set of lines (respectively points) in the dual projective
plane.

The operator $\L_{\{2\},\{3\}}$ is therefore the map which to a line
arrangement $\cA$ returns the union of all the lines containing exactly
three double points of $\cA$. 

A line arrangement $\cA$ of six lines in general position has $15$
double points and the line arrangement $\check{\cA}=\cL_{2}(\cD(\cA))$
(the union of the lines going through at least two points in the dual
set $\cD(\cA)$), is the union of $15$ lines with $45$ double points
and six $5$-points (the points in $\cD(\cA)$), ie. we have $t_{2}(\check{\cA})=45,\,t_{5}(\check{\cA})=6$.
That leads us to the following definition:
\begin{defn*}
We say that an arrangements $\cC_{0}$ of six lines in the projective
plane is \textit{unassuming} if the singularities of $\cC_{0}$ are
$15$ double points and the line arrangement $\check{\co}=\cL_{2}(\cD(\cC_{0}))$
has singularities 
\[
t_{2}(\check{\co})=27,\,t_{3}(\check{\co})=6,\,t_{5}(\check{\co})=6,
\]
we require moreover that the six $5$-points of $\check{\co}$ are
not contained on a conic. 
\end{defn*}
We use the word ``unassuming'' in order to describe something that
is deceptively simple (here $15$ nodal points) but has hidden qualities
or advantages (the dual configuration is not generic). As it will
be clear, one needs the value $t_{3}=6$, because we are dealing with
$6$ lines: the dual of these $6$ triple points creates six new lines.
 We work over the complex numbers; we obtain the following result.
\begin{thm}
\label{thm:Main1}If $\co$ is a generic unassuming arrangement, then
$\L_{\{2\},\{3\}}(\co)$ is also an unassuming arrangement. The (closure
of the) moduli space of unassuming arrangements has two irreducible
components, one is isomorphic to $\PP^{1}$, the other one is a point. 
\end{thm}

We construct explicitly the unassuming arrangements $\cC_{0}=\cC_{0}(\nt)$,
which are parametrized by $\nt$ in $\PP^{1}\setminus S$, where $S=\{\infty,-1,0,1,\pm2\pm\sqrt{5}\}$.
The genericity assumption in Theorem \ref{thm:Main1} comes from the
fact that for finitely many $\nt\notin S$, the image $\L_{\{2\},\{3\}}(\co(z_{o}))$
may correspond to a point in $S$. 

From the above result, starting with an unassuming arrangement $\cC_{0}=\cC_{0}(\nt)$,
one can define by induction a sequence $(\cC_{k})_{k\in\NN}$ of unassuming
arrangements by setting $\cC_{k+1}=\L_{\{2\},\{3\}}(\cC_{k})$. It
is then natural to study these sequences $(\cC_{k})_{k\in\NN}$ and
to expect interesting configurations from their union, in particular
when the sequence is periodic. We obtain the following result.
\begin{thm}
\label{thm:Main2}For fixed $\nt$ in the complement of the unit circle
and any integers $k\neq k'$, the unassuming arrangements $\cC_{k},\cC_{k'}$
are not projectively equivalent. \\
For $\nt$ a primitive $n$-root of unity with $n>1$ odd, the associated
sequence $(\cC_{k})_{k\in\NN}$ with $\cC_{0}=\cC_{0}(\nt)$ is periodic.
For $n\in\{3,5,\dots,21,23\}$, the union of the periodic line arrangements
$(\cC_{k})_{k\in\NN}$ is a Ceva line arrangement or is contained
in a Ceva line arrangement. \\
The line operator $\L_{\{2\},\{3\}}$ acts on the moduli space of
unassuming arrangement (isomorphic to $\PP^{1}(z)$) through the second
Tchebychev polynomial map $z\to2z^{2}-1$.
\end{thm}

We recall that the Ceva$(n)$ line arrangement is (up to projective
equivalence) the union of the following $3n$ lines: 
\[
(x^{n}-y^{n})(x^{n}-z^{n})(y^{n}-z^{n})=0
\]
in the projective plane; this is a rather prominent line arrangement
e.g., for the construction of ball quotient surfaces as cover of the
plane, or for studying Terao's freeness conjecture. In characteristic
$0$, this is the unique known infinite family of Sylvester-Gallai
line arrangements. We do not known if an $n$-periodic unassuming
line arrangement is always contained a Ceva line arrangement, (see
Section \ref{subsec:Periodic-points} where other cases $n$ are discussed).

By duality, one can rephrase the results in Theorem \ref{thm:Main1}
on the action of $\L_{\{2\},\{3\}}$ on unassuming arrangements as
follows:
\begin{thm}
\label{thm:Main3}For a set $P_{6}=\{p_{1},\dots,p_{6}\}$ of six
points in the plane, consider the following property:
\begin{lyxlist}{00.00.0000}
\item [{$(P)$}] The union of the lines containing two points in $P_{6}$
possesses exactly six triple points $p_{1}',\dots,p_{6}'$. The points
of $P_{6}$ are not inscribed in a conic.
\end{lyxlist}
Suppose that $P_{6}$ is generic to satisfy satisfies $(P)$. Then
the set of triple points $P_{6}'=\{p_{1}',\dots,p_{6}'\}$ also satisfies
$(P)$, moreover if the points in $P_{6}$ are real, there exists
a unique set of six real points $P_{6}^{-}$ satisfying (P) and such
that $(P_{6}^{-})'=P_{6}$.
\end{thm}

See Figure \ref{fig:The-dual-configuration} for an example of such
set of points $P_{6}$ (in black) and its associated set $P_{6}'=\Psi_{\{2\},\{3\}}(P_{6})$
(in red); joining the red points, one can find the set $P_{6}''$
(however, it seems more difficult to find $P_{6}^{-}$ using only
geometry). Depending on the situation, for example if the points in
$P_{6}$ have real coordinates, the sequence $P_{6},P_{6}',P_{6}''$
etc. has only distinct elements. In contrast, for any integer $n>1$,
there exist sequences $P_{6},P_{6}',P_{6}''...$ which are periodic
with period at least $n$. 

Our proof of Theorem \ref{thm:Main3} is analytic. For a discussion
about Theorem \ref{thm:Main3} and classical geometry, such as Pascal's
hexagon Theorem, see \cite{OSO}. %
{} 

\begin{figure}[h]
\begin{center}

\begin{tikzpicture}[scale=3]

\clip(0.8386887503215222,0.4097461379762637) rectangle (3.181712280433072,2.848605609375874);
\draw [domain=0.8386887503215222:3.181712280433072] plot(\x,{(-1.2360680732947948--2.*\x)/1.2360679});
\draw [line width=0.2mm] (1.8944271599999993,0.4097461379762637) -- (1.8944271599999993,2.848605609375874);
\draw [domain=0.8386887503215222:3.181712280433072] plot(\x,{(--1.788854285341042-0.34164074000000033*\x)/0.5527864199999999});
\draw [domain=0.8386887503215222:3.181712280433072] plot(\x,{(--2.3105973134818214E-8-0.1842621400000004*\x)/-0.29814238666666704});
\draw [domain=0.8386887503215222:3.181712280433072] plot(\x,{(-0.3685242684470138--0.2981423866666667*\x)/0.11388024666666663});
\draw [domain=0.8386887503215222:3.181712280433072] plot(\x,{(--1.541640726136417-0.*\x)/0.89442716});
\draw [domain=0.8386887503215222:3.181712280433072] plot(\x,{(--2.5527863506820836-0.34164074000000033*\x)/1.10557284});
\draw [domain=0.8386887503215222:3.181712280433072] plot(\x,{(-0.21114571465895904--1.1055728400000002*\x)/1.44721358});
\draw [domain=0.8386887503215222:3.181712280433072] plot(\x,{(--0.4721358000000002--0.7639320999999999*\x)/2.});
\draw [domain=0.8386887503215222:3.181712280433072] plot(\x,{(--2.341640774658959-1.44721358*\x)/-0.3416407400000001});
\draw [domain=0.8386887503215222:3.181712280433072] plot(\x,{(--3.0786892422350673-1.0786893000000002*\x)/0.25464403333333285});
\draw [domain=0.8386887503215222:3.181712280433072] plot(\x,{(-1.3167183813641654--0.34164074000000033*\x)/-0.21114568});
\draw [domain=0.8386887503215222:3.181712280433072] plot(\x,{(-1.0786892999999997-0.2546440333333333*\x)/-1.3333333333333333});
\draw [domain=0.8386887503215222:3.181712280433072] plot(\x,{(-0.6666667244315985--1.3333333333333333*\x)/1.0786892999999997});
\draw [domain=0.8386887503215222:3.181712280433072] plot(\x,{(-0.40000000693179194--0.5527864200000003*\x)/0.5527864200000001});
\begin{scriptsize}
\draw [fill=black] (2.44721358,1.7236067900000003) circle (0.2mm);
\draw [fill=black] (2.2360679,2.61803395) circle (0.2mm);

\draw [fill=black] (1.89442716,1.7236067900000003) circle (0.2mm);
\draw [fill=black] (3.,1.38196605) circle (0.2mm);
\draw [fill=black] (1.,0.61803395) circle (0.2mm);
\draw [fill=black] (1.89442716,1.17082037) circle (0.2mm);
\draw [fill=red] (2.0786892999999997,1.9513672833333333) circle (0.2mm);
\draw [fill=red] (2.490711933333333,1.5393446499999999) circle (0.2mm);
\draw [fill=red] (1.5527864200000001,0.82917963) circle (0.2mm);
\draw [fill=red] (1.6666666666666667,1.1273220166666666) circle (0.2mm);
\draw [fill=red] (1.89442716,2.06524753) circle (0.2mm);
\draw [fill=red] (2.78885432,1.7236067900000003) circle (0.2mm);
\end{scriptsize}

\end{tikzpicture}

\end{center} 

\caption{\label{fig:The-dual-configuration}The points in $P_{6}$ and the
six associated triple points}
\end{figure}
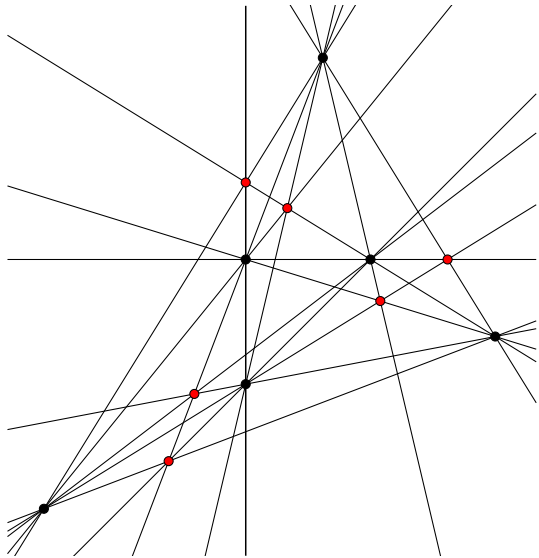

In \cite{SchwartzPent}, Schwartz defines the pentagram map: it is
the map which to a polygon of $m\geq5$ lines i.e., a $m$-tuple of
lines $(\ell_{1},\dots,\ell_{m})$, associate the polygon $(\ell_{1}',\dots,\ell_{m}')$,
where $\ell_{k}'$ is the line containing the points $p_{k},p_{k+2}$,
for $p_{k}=\ell_{k}\cap\ell_{k+1}$ (the indices being taken mod $m$).
That rational self-map on the space of $m$-gons is of interest in
classical projective geometry, algebraic combinatorics, moduli spaces,
cluster algebras, and it is an example of a Liouville-Arnold discrete
integrable system \cite{Berger}, \cite{OST}. Our paper is an example
of another type of operators acting on line arrangements. %

Let us now review the structure of the paper, and some of the content
not mentioned in the above Theorems. 

In Section \ref{sec:unassuming-arrangements-and}, we construct the
moduli space of labelled (or marked) unassuming arrangements. We obtain
that the image by the operator $\ldt$ of an unassuming arrangement
is again an unassuming arrangement, and we compute the action of the
operator $\ldt$ on the moduli space without marking. Since we are
interested in the lines arrangements rather than their class modulo
projective equivalence, we compute the formulas for the equations
of the lines of the successive images of the operator $\ldt$. If
one starts from the configuration $\co=\co(\nt)$ (for some fixed
$\nt\in\PP^{1}$) and inductively defines the sequence $(\cC_{k})_{k\in\NN}$
by $\cC_{k+1}=\L_{\{2\},\{3\}}(\cC_{k})$, the coordinates of the
line arrangement $\cC_{k}$ are controlled by the powers $\nt^{\nu_{k}}$,
where $(\nu_{k})_{k\geq0}$ is the sequence beginning by $0,1,1,3,5,11...$
and defined by 
\[
\nu_{0}=0,\,\,\,\nu_{2k+1}=2\nu_{2k}+1,\,\,\,\nu_{2k}=2\nu_{2k-1}-1.
\]
By contrast, the action of $\ldt$ on the moduli space of unmarked
line arrangements up to projectivity is through iterate of the Tchebychev
map $F_{\L}:\nt\to2\nt^{2}-1$. We discuss also in this section if
the action of $\ldt$ on unassuming line arrangement can be obtained
by a Cremona transformation. 

In Section \ref{sec:dynamics-and-Applications}, after recalling some
definitions on dynamical system, we study the periodic configurations.
The sequence $(\cC_{k})_{k\in\NN}$ starting with $\co=\co(\nt)$
is periodic with period dividing $n>2$ if $\nt^{\nu_{n}}=1$. We
give examples of periodic line arrangements, and show that their union
is strongly related to the Ceva arrangements. We also discuss the
Julia and Fatou domains of $F_{\L}$ and their meaning for $\ldt$.
We then study the isolated $0$ dimensional moduli space of unassuming
arrangements. We obtain that the corresponding line arrangement is
related to the famous Hesse arrangement of $12$ lines. 

The six points in $P_{6}$ which are dual to an unassuming arrangement
are not contained in a conic, thus by blowing up these points, we
can associate to $P_{6}$ a smooth cubic surface $X$. Then the triple
points of $\cL_{2}(P_{6})$ have a geometric meaning for $X$: these
points are Eckardt points, i.e., they are isolated fixed points of
involutions acting on $X$. Using the classification of the automorphisms
group of cubic surfaces, we then show that we obtained all unassuming
arrangements. 

Finally, let us mention that the K3 surfaces obtained as double cover
of the plane ramified over an unassuming arrangement have Picard number
$\geq19$; these surfaces and how they are related through the operator
$\ldt$ are the object of a paper in preparation \cite{RS}. 

We used the used the computer algebra system Magma \cite{Magma};
the Magma code for dealing with matroids, the unassuming line arrangements
and the various operators can be found on arXiv, in the ancillary
file of this paper.

\textbf{Acknowledgements} The author is grateful to Lukas Kühne for
his explanations on the realization of matroids and pointing out the
paper \cite{ACKN}; he is also grateful to the referees for their
comments improving the paper. Support from Centre Henri Lebesgue ANR-11-LABX-0020-01
is acknowledged.

\section{\label{sec:unassuming-arrangements-and}unassuming arrangements and
the $\protect\L_{\{2\},\{3\}}$-operator}

\subsection{\label{subsec:Matroids-and-six}Matroids and moduli space of labelled
unassuming arrangements}

 In the present section, we construct three irreducible components
of the moduli space of unassuming arrangements with a labeling, by
which we mean that the elements of the moduli are $6$-tuples 
\[
\cA=(\ell_{1},\dots,\ell_{6})
\]
of lines. One irreducible component, which we denote by $\mathfrak{B}_{6}$,
is one dimensional and for any arrangement $\cC_{0}$ in $\mathfrak{B}_{6}$,
and up to automorphism of $\PP^{2}$, there exists $\nt\in\PP^{1}(\CC)$
with $\nt\notin S=\{\infty,-1,0,1,\pm2\pm\sqrt{5}\}$ such that the
normals of the six lines lines $\ell_{1},\dots,\ell_{6}$ of $\cC_{0}=\cC_{0}(\nt)$
are given by the six columns of the following matrix
\begin{equation}
M_{\nt}=\left(\begin{array}{c}
\begin{array}{cccccc}
1 & 0 & 0 & 1 & \nt+1 & -\nt+1\\
0 & 1 & 0 & 1 & -\nt+1 & \nt+1\\
0 & 0 & 1 & 1 & 2 & 2
\end{array}\end{array}\right).\label{eq:Matrix-Mt}
\end{equation}
Case $\nt\in\{\infty,0\}$ is degenerate: the arrangement associated
to $M_{\nt}$ has $5$ lines with one triple point. Case $\nt=\pm1$
gives the complete quadrilateral: the six lines going through $4$
points in general position. The line arrangement with $\nt\in\{\pm2\pm\sqrt{5}\}$
given by matrix $M_{\nt}$ has six lines with $15$ nodes but is not
an unassuming arrangement; we discuss that case in the last Section. 

As we show below, the two other irreducible components of the labeled
moduli space are points which are conjugated under the Galois action,
and the two associated line arrangements are projectively equivalent
when one forgets the labelings.

In order to construct the moduli space $\mathfrak{B}_{6}$, we will
use some vocabulary and elementary technics of matroid theory. A matroid
is a pair $(E,\cB)$, where $E$ is a finite set and the elements
of $\cB$ are subsets of $E$, called \textit{basis,} subject to the
following properties:\\
$\bullet$ $\cB$ is non-empty\\
$\bullet$ if $A$ and $B$ are distinct members of $\cB$ and $a\in A\setminus B$,
then there exists $b\in B\setminus A$ such that $(A\setminus\{a\})\cup\{b\}\in\cB$. 

The bases have the same order $n$, called the rank of $(E,\cB)$.
Order $n$ subsets of $E$ that are not bases are called \textit{non-bases}.
We identify $E$ with $\{1,\dots,m\}$. A realization (over some field)
of the matroid $(E,\cB)$ is  the data of a size $n\times m$ matrix
$M$, with columns $C_{1},\dots,C_{m}$, such that any order $n$
subset $\{i_{1},\dots,i_{n}\}$ of $E$ is a non-base if and only
if the size $n$ minor $|C_{i_{1}},\dots,C_{i_{n}}|$ is zero. 

For our problem, we identify the lines of the searched arrangement
to $E=\{1,\dots,m\}$. Since we search for line arrangements in $\PP^{2}$,
one has $n=3$: the columns of the matrix $M$ are the normals of
the lines. Three columns $(i,j,k)$ are non-base if and only if the
lines $i,j,k$ meet at a point, indeed if $(x_{1i},x_{2i},x_{3i}),\dots,(x_{1k},x_{2k},x_{3k})$
are the normals of lines $i,j,k$, the lines $(i,j,k)$ meet at point
$(v_{1}:v_{2}:v_{3})$ implies that the matrix $(x_{s,t})_{1\leq s\leq3,t\in\{i,j,k\}}$
has a non-trivial kernel, and conversely.

We are looking for line arrangements of $15$ lines with singularities
$t_{2}=27,\,t_{3}=6,\,t_{5}=6.$ More precisely, we impose that the
following elements are non-bases
\begin{equation}
\text{NB}_{1}:\,\,\,(1,5,10),(2,4,13),(3,9,13),(5,7,8),(6,11,12),(12,14,15)\label{eq:Line6}
\end{equation}
(these are the triple of lines meeting at the same points), moreover
any triple of lines among the following quintuples 
\begin{equation}
\begin{array}{c}
(1,2,8,9,12),\,(1,6,7,13,15),\,(2,3,5,6,14),\\
(3,4,7,10,12),\,(4,5,9,11,15),\,(8,10,11,13,14),
\end{array}\label{eq:Quintuples}
\end{equation}
must be non-basis (these are the six quintuples of lines meeting at
the $5$-points). In order to obtain the moduli space of such line
arrangement, one takes a $3\times15$ matrix $T$, with entries as
unknowns. One can reduce the number of variables, by taking the four
first columns to be the canonical basis of $\PP^{2}$, ie $C_{1}=(1,0,0)$,
$C_{2}=(0,1,0)$, $C_{3}=(0,0,1)$, $C_{4}=(1,1,1)$ (taking care
that the four minors of $(1,2,3,4)$ are bases of the matroid). Since
the columns are the normals of the lines, one can suppose that in
the remaining columns there is always a $1$ as follows: If the set
$(2,3,i)$ is a basis, then the first entry in row $i$ must be non-zero,
so we assume it\textquoteright s $1$. If not, the first-row entry
in column $i$ must be zero and we fix the second-row entry to be
$1$ (it cannot be $0$ since we would then have a repetition of the
normal $(0,0,1)$). The ideal generated by the $66$ minors corresponding
to the non-bases, defines a variety in $\mathbb{A}^{19}$, which is
our moduli space of $15$ lines arrangements with singularities $t_{2}=27,\,t_{3}=6,\,t_{5}=6.$
One then takes the dual of the six $5$-points and obtain (using the
computer algebra system Magma) the moduli space $\mathfrak{B}_{6}$
(isomorphic to $\PP^{1}$), with the normals in \eqref{eq:Matrix-Mt},
which is what we call the (labelled) moduli space of unassuming arrangements.
The entries of the matrix $T$ are rational functions in the variable
$\nt$. One can compute the determinants of the $3\times3$ minors
of it, their numerators (resp. denominators) are powers and products
of elements in $\{\nt,\nt\pm1,\nt^{2}\pm4\nt-1\}$ (resp. $\{\nt,\nt\pm1\}$).
The vanishing of such additional minors create other singularities,
or gives a degeneration of the arrangement i.e., an arrangement with
repetition of some lines, thus the condition $\nt\notin\{\infty,-1,0,1,\pm2\pm\sqrt{5}\}$.

A difficulty (which does not appear in this paper) in computing the
moduli space of line arrangements was to find the appropriate parametrization
of it $\g:\PP^{1}\stackrel{\simeq}{\to}\mathfrak{B}_{6}$, so that
one can understand clearly the action of the line operator $\L_{\{2\},\{3\}}$
on $\kBS$ and on the plane, and so that one can understand the associated
dynamical systems. For example, the well-known Techbychev map $F_{\L}:z\to2z^{2}-1$
will appear, but our first parametrization gave instead the map $\tilde{F}:z\to\tfrac{z^{2}-16}{4z^{2}}$.
Using the theory of multipliers of rational functions (see e.g., \cite{Silverman}),
we recognized the dynamic of $F_{\L}$ and found another parametrization
$\phi:\PP^{1}\to\PP^{1}$, so that $\phi^{-1}\circ\tilde{F}\circ\phi$
equals to $F_{\L}$. The moduli space $\kBS$ is defined over $\QQ$. 

There are two other moduli spaces of labeled unassuming arrangements,
namely, one can take for non-bases the following triple (which correspond
to the six triple points): 
\begin{equation}
\text{NB}_{2}\,\,\,(1,3,11),(1,4,14),(3,8,15),(4,6,8),(6,11,12),(12,14,15)\label{eq:triple2}
\end{equation}
and the same $60$ triple coming from equation \eqref{eq:Quintuples}.
One obtains that the associated moduli space, the realization space
of that matroid, is the union of two points. We will discuss that
case in Section \ref{subsec:Disassembling-Hesse} and in Section \ref{subsec:The-cubic-surface}. 

We will prove in Section \ref{subsec:The-cubic-surface} that there
are no other irreducible components of the moduli space of unassuming
arrangements, or in other words that we can always label an unassuming
arrangement in a way that leads to Equations \eqref{eq:Quintuples}
and \eqref{eq:Line6} or \eqref{eq:triple2}. We used the computer
algebra system Magma \cite{Magma} for the computations of these moduli
spaces and double-checked using the Homalg package \cite{Homalg}
developed in Julia/Oscar. 

We explain in Section \ref{subsec:Free-arrangements} how we found
these two matroids.

\subsection{\label{subsec:Images of six lines}Image of an unassuming arrangement
by the $\protect\L_{\{2\},\{3\}}$-operator}

Let us fix the parameter $\nt$ and apply the $\L_{\{2\},\{3\}}$-operator
on the line arrangement $\cC_{0}(\nt)\in\mathfrak{B}_{6}$. The $15$
nodal points of $\cC_{0}$ are given by the columns of the following
matrix:
\[
\left(\begin{array}{ccccccccccccccc}
1 & 0 & 0 & -1 & -1 & -1 & 0 & 0 & -2 & 2 & \nt-1 & 0 & \nt+1 & \nt-1 & \nt+1\\
0 & 1 & 0 & -1 & 0 & 1 & -1 & -2 & 0 & 0 & \nt+1 & 2 & \nt-1 & \nt+1 & \nt-1\\
0 & 0 & 1 & 1 & 1 & 0 & 1 & \nt+1 & \nt+1 & \nt-1 & 0 & \nt-1 & -2\nt & -2\nt & 0
\end{array}\right),
\]
these columns are also the normals of the $15$ lines in the dual
configuration. The image of $\cC_{0}$ by the point operator $\cP_{\{2\}}$
is the union of the above $15$ nodal points of $\cC_{0}$. By construction
of $\cC_{0}$, the line arrangement $\check{\cC}_{0}=\cL_{2}\cD(\cC_{0})$
has six triple points, dually that means that there are exactly six
lines in $\PP^{2}$ such that each line contain exactly three points
in $\cP_{\{2\}}(\cC_{0})$, therefore 
\[
\cC_{1}=\cL_{\{3\}}\cP_{\{2\}}(\cC_{0})=\L_{\{2\},\{3\}}(\cC_{0})
\]
is a line arrangement of six lines. The interesting fact is that the
situation is self-similar: $\cC_{1}$ is also an unassuming arrangement.
Indeed, a direct computation shows that the normals of the line arrangement
$\cC_{1}=\L_{\{2\},\{3\}}(\cC_{0})$ are 
\[
\left(\begin{array}{c}
\begin{array}{cccccc}
1+\nt & 1-\nt & -\nt+1 & \nt+1 & 1 & 0\\
1-\nt & 1+\nt & -\nt+1 & \nt+1 & 0 & 1\\
0 & 0 & 2 & 2 & 1 & 1
\end{array}\end{array}\right),
\]
moreover there exists a unique projective transformation such that
the six lines of $\cC_{1}$ are sent respectively to the lines with
normals
\begin{equation}
\left(\begin{array}{c}
\begin{array}{cccccc}
1 & 0 & 0 & 1 & \nt^{2}+1 & \nt^{2}-1\\
0 & 1 & 0 & 1 & \nt^{2}-1 & \nt^{2}+1\\
0 & 0 & 1 & 1 & 2\nt^{2} & 2\nt^{2}
\end{array}\end{array}\right),\label{eq:cC_1-in-B6}
\end{equation}
which is the line arrangement $\cC_{0}(\tfrac{-1}{\nt^{2}})$ in
$\kBS$ (compare with matrix $M_{\nt}$ in \eqref{eq:Matrix-Mt}).
From that description, we see that.
\begin{cor}
\label{cor:For--such}For $\nt\not\in S=\{\infty,-1,0,1,\pm2\pm\sqrt{5}\}$
such that $-\tfrac{1}{\nt^{2}}\notin S$, the image of the unassuming
arrangement $\cC_{0}(\nt)$ by $\ldt$ is an unassuming arrangement.
The operator $\L_{\{2\},\{3\}}$ acts on the moduli of labelled line
arrangements $\kBS$ by the map $\cC_{0}(\nt)\to\cC_{0}(-\tfrac{1}{\nt^{2}})$.
\end{cor}

One may generalize the first part of Corollary \ref{cor:For--such},
since any unassuming line arrangement $\cC$ is always projectively
equivalent to a line arrangement $\cC_{0}(\nt)$, $\nt\notin S$.
Also, as we will see in Section \ref{tab:First-periodic-configurations},
the image of an unassuming arrangement in the moduli with two points
is also an unassuming arrangement, and in the same family.

See Figure \ref{fig:The-povera-arr} for a picture of an unassuming
arrangement $\co$ (the lines in black) and its image $\cC_{1}$ (the
blue lines) by $\ldt$. It is the image of an unassuming arrangement
$\cC_{0}(\nt)$ by the projective transformation sending two of the
three base points (see Definition \ref{def:Triangle}) to infinity. 

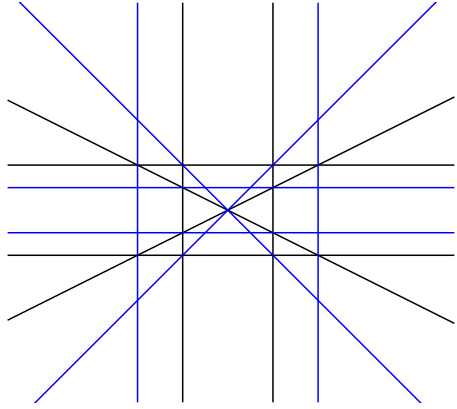
\begin{figure}[h]
\begin{center}

\begin{tikzpicture}[scale=0.6]

\clip(-4.88,-4.26) rectangle (5.02,4.6);
\draw [line width=0.2mm,domain=-4.88:5.02] plot(\x,{(-0.-0.5*\x)/1.});
\draw [line width=0.2mm] (1.,-4.26) -- (1.,4.6);
\draw [line width=0.2mm] (-1.,-4.26) -- (-1.,4.6);
\draw [line width=0.2mm,domain=-4.88:5.02] plot(\x,{(-0.--0.5*\x)/1.});
\draw [line width=0.2mm,domain=-4.88:5.02] plot(\x,{(-1.-0.*\x)/1.});
\draw [line width=0.2mm,domain=-4.88:5.02] plot(\x,{(-1.-0.*\x)/-1.});
\draw [line width=0.2mm,color=blue,domain=-4.88:5.02] plot(\x,{(-0.-1.*\x)/-1.});
\draw [line width=0.2mm,color=blue,domain=-4.88:5.02] plot(\x,{(-0.-1.*\x)/1.});
\draw [line width=0.2mm,color=blue,domain=-4.88:5.02] plot(\x,{(-0.5-0.*\x)/1.});
\draw [line width=0.2mm,color=blue] (-2.,-4.26) -- (-2.,4.6);
\draw [line width=0.2mm,color=blue,domain=-4.88:5.02] plot(\x,{(--0.5-0.*\x)/1.});
\draw [line width=0.2mm,color=blue] (2.,-4.26) -- (2.,4.6);

\end{tikzpicture}

\end{center} 

\caption{\label{fig:The-povera-arr}An unassuming arrangement and its image
by $\protect\ldt$}
\end{figure}

It is also interesting to consider the dual configuration $\cD(\cC_{0})$.
The arrangement $\check{\cC_{0}}=\cL_{2}(\cD(\cC_{0}))$, obtained
by taking the union of all lines through the six points of $\cD(\cC_{0})$
contains $15$ lines, and has singularities
\[
t_{2}=27,\,t_{3}=6,\,t_{5}=6.
\]
The arrangement $\check{\cC_{1}}=\cL_{2}(\cD(\cC_{1}))$ has the same
properties, moreover the following relation holds
\[
\check{\cC_{1}}=\L_{\{3\},\{2\}}(\check{\cC_{0}}).
\]
For $k\in\{0,1\}$, the six $5$-points of $\check{\cC_{k}}$ correspond
by duality to the six lines of $\cC_{k}$, and the $5$-points of
$\check{\cC_{1}}$ are the $3$-points of the arrangement $\check{\cC_{0}}$.
In Figure \ref{fig:The-dual-configuration}, the points $p_{1},\dots,p_{6}$
of $\cD(\cC_{0})$ are in black (these are the $5$-points of $\check{\cC_{0}}$),
the triple points of $\check{\cC_{0}}$ are in red (these are the
$5$-points of $\check{\cC_{1}}$), the $15$ black lines are the
lines of $\check{\cC_{0}}$ (here for some fixed value $\nt$).

The matrix $M_{\nt}$ in \eqref{eq:Matrix-Mt} giving the normals
of the six lines attributes implicitly a marking, i.e., an order to
the six lines of the arrangement $\cC_{0}=\cC_{0}(\nt)$ and the parameter
space $\mathfrak{B}_{6}\simeq\PP^{1}$ is a moduli space for marked
line arrangements $(\cC_{0}(\nt),\ell_{1},\dots,\ell_{6})$. In the
next section we describe the forgetful map $(\cC_{0}(\nt),\ell_{1},\dots,\ell_{6})\to\cC_{0}(\nt)$,
which takes the quotient of $\kBS$ by the permutation group $S_{6}$
of six elements. 

\subsection{The moduli space of unassuming arrangements without marking}

Let $(\cC_{0}(\nt),\ell_{1},\dots,\ell_{6})\in\mathfrak{B}_{6}$.
Consider another order $\ell_{\s(1)},\dots,\ell_{\s(6)}$, for $\s\in S_{6}$
a permutation. Since the lines $\ell_{i}$ are in general position,
the lines $\ell_{\s(1)},\dots,\ell_{\s(4)}$ form a basis in the projective
space, and by using an appropriate base-change, the normal of the
lines $\ell_{\s(1)},\dots,\ell_{\s(6)}$ are the six columns of the
matrix
\[
M_{\s}(\nt)=\left(\begin{array}{c}
\begin{array}{cccccc}
1 & 0 & 0 & 1 & \a_{1,5} & \a_{1,6}\\
0 & 1 & 0 & 1 & \a_{2,5} & \a_{2,6}\\
0 & 0 & 1 & 1 & \a_{3,5} & \a_{3,6}
\end{array}\end{array}\right)
\]
where the $\a_{i,j}$ are fractions in $\nt$. One can then define
a morphism 
\[
\psi_{\s}:\nt\in\PP^{1}\to((\a_{1,5}:\a_{2,5}:\a_{3,5}),(\a_{1,6}:\a_{2,6}:\a_{3,6}))\in\PP^{2}\times\PP^{2}.
\]
One computes that the images of $\PP^{1}=\PP^{1}(\nt)$ by these maps
are $15$ curves in $\PP^{2}\times\PP^{2}$, among which there is
a unique curve with bi-degree $(1,1)$, which we denote by $L_{0}$.
There is a sub-group of $48$ permutations $\s\in S_{6}$ such that
the image of $\psi_{\s}$ is the line $L_{0}$. For these $48$ permutations,
the matrix $M_{\s}(\nt)$ is equal to either
\[
M(\nt),M(-\nt),M(\tfrac{1}{\nt})\text{ or }M(-\tfrac{1}{\nt}).
\]
Let $G_{4}$ be the order $4$ automorphism group of $\PP^{1}$ generated
by the involutions $\s_{1}:\nt\to-\nt$ and $\s_{2}:\nt\to1/\nt$.
The degree $4$ map 
\begin{equation}
\varUpsilon:\PP^{1}\to\PP^{1},\,\,\nt\to\tfrac{1}{2}(\nt^{2}+\nt^{-2})\label{eq:gamma-1}
\end{equation}
is invariant by the group $G_{4}$ and therefore realizes the quotient
of $\PP^{1}$ by $G_{4}$. The quotient of $S_{6}$ by the order $48$
group has $15$ classes, which classes correspond to the $15$ curves
contained in $\PP^{2}\times\PP^{2}$. Thus the map $\varUpsilon:\PP^{1}\to\PP^{1}$
is also the quotient map $\phi:\kBS\to\kBS/S_{6}$; we call it the
period map. The images by $\varUpsilon$ of $\infty,-1,0,1$ are
respectively $\infty,1,\infty,1$. 
\begin{defn}
The operator $\L_{\{2\},\{3\}}$ has a natural action on the moduli
space $\kBS/S_{6}=\PP^{1}(z)$: we denote by $F_{\L}:\PP^{1}\to\PP^{1}$
the map obtained by sending the point $z\in\kBS/S_{6}$ representing
the unassuming arrangement $\cC_{0}(z)$, to the point $F_{\L}(z)\in\kBS/S_{6}$
representing the line arrangement $\cC_{1}=\L_{\{2\},\{3\}}(\cC_{0})$.
\end{defn}

\begin{thm}
\label{Thm-action-of-FL}The action of the operator $\L_{\{2\},\{3\}}$
on the moduli space $\kBS/S_{6}=\PP^{1}(z)$ is given by the polynomial
map
\[
F_{\L}:z\to2z^{2}-1.
\]
\end{thm}

\begin{proof}
From Equation \eqref{eq:cC_1-in-B6}, there exists a projective transformation
$\text{P}$ such that $\text{P}(\cC_{1}(\nt))$ is the labelled arrangement
$\cC_{0}(\tfrac{-1}{\nt^{2}})$. The image by the quotient map $\phi$
of the labelled arrangement $\text{P}(\cC_{1}(\nt))$ is the point
\[
\phi(\cC_{0}(\tfrac{-1}{\nt^{2}}))=\varUpsilon(\tfrac{-1}{\nt^{2}})=\tfrac{1}{2}(\nt^{4}+\nt^{-4})=2\left(\tfrac{1}{2}(\nt^{2}+\nt^{-2})\right)^{2}-1=F_{\L}(\tfrac{1}{2}(\nt^{2}+\nt^{-2})),
\]
thus the result.
\end{proof}
The moduli space $\kBS/S_{6}$ is a moduli space of line arrangement
up to projective transformation: there is a projective transformations
sending an unassuming arrangement $\cC$ to $\cC'$ if and only if
the images of $\cC$ and $\cC'$ in $\kBS/S_{6}$ are equal. That
has the consequence:
\begin{cor}
There exists unassuming arrangements $\cC_{0}$ such that the orbit
$(\cC_{k})_{k\in\NN}$ of $\cC_{0}$ by $\L_{\{2\},\{3\}}$ is infinite,
and there are no projective transformations between $\cC_{k},\cC_{k'}$
for any $k\neq k'$.
\end{cor}

Indeed, for $\cC_{0}=\co(\nt)$, one just have to take $\nt\in\PP^{1}$
such that the sequence $F_{\L}^{\circ n}(\nt)$ tends to $\infty$.
\begin{table}

\caption{}

\end{table}

\subsection{The coordinates of the iterations of unassuming arrangements}

In this section, in order to understand the dynamics of the operator
$\L_{\{2\},\{3\}}$ on unassuming arrangements, not up to projective
automorphism of $\PP^{2}$ but in the plane, we would like to know
what are the equations of the six lines of the successive images $\cC_{k},k=1,2,...$
of an unassuming arrangement $\cC_{0}$ by $\L_{\{2\},\{3\}}$. That
will be used in Section \ref{subsec:Periodic-points}, where we study
periodic and pre-periodic unassuming arrangements for the action of
$\L_{\{2\},\{3\}}$.

Consider the arrangement of lines $\cC(a,b,c)$ whose normals are
the columns of the following matrix
\begin{equation}
\left(\begin{array}{cccccc}
0 & 0 & b & -b & c & -c\\
a & -a & 0 & 0 & 1 & 1\\
1 & 1 & 1 & 1 & 0 & 0
\end{array}\right).\label{eq:abc}
\end{equation}
For $a,b,c$ generic, it is projectively equivalent to $\cC_{0}(\frac{ac}{b})$,
and therefore it is an unassuming arrangement. One computes that the
image of $\cC(a,b,c)$ by $\ldt$ is the unassuming arrangement $\cC(\frac{b}{c},ac,\frac{b}{a})$. 

Rather than computing the coordinates of the lines of the sequence
associated to $\co$, let us compute those associated to the arrangement
$\tilde{\co}$, where 
\begin{equation}
\tilde{\co}=\cC(1,1,\nt),\label{eq:NewEqPov}
\end{equation}
since that choice greatly simplify the formulas and computations. 

Let us define the integer sequence $(\nu_{k})_{k\geq0}$ by $\nu_{k}=\frac{1}{3}(2^{k}-(-1)^{k})$
for $k\in\NN$. That sequence may also be inductively defined by 

\[
\nu_{0}=0,\,\nu_{k+1}=2\nu_{k}+(-1)^{k+1}
\]
It is the OEIS sequence $A001045$ and it begins with $0,1,1,3,5,11,21,43...$
Let us also define the sequences $(a_{k})_{k\geq0},(b_{k})_{k\geq0},(c_{k})_{k\geq0}$
by 
\[
a_{k}=\nt^{\nu_{k}(-1)^{k}},\,\,b_{k}=\nt^{\nu_{k}(-1)^{k+1}}=\frac{1}{a_{k}},\,\,c_{k}=\nt^{\nu_{k}(-1)^{k}+1}=\nt a_{k}.
\]
Let $(\tilde{\cC_{k}})$ be the sequence defined by $\tilde{\co}=\cC(1,1,\nt)$
and $\tilde{\cC}_{k+1}=\ldt(\tilde{\cC}_{k})$. 
\begin{thm}
\label{thm:Formula-ABC}We have $\tilde{\cC}_{k}=\cC(a_{k},b_{k},c_{k})$.
\end{thm}

\begin{proof}
This is true when $k=0$. Let $k\in\NN$ be an integer, and suppose
that this is true for $k$. One has 
\[
\begin{array}{c}
\frac{b_{k}}{c_{k}}=\frac{\nt^{\nu_{k}(-1)^{k+1}}}{\nt^{\nu_{k}(-1)^{k}+1}}=\nt^{(-1)^{k+1}(2\nu_{k}+(-1)^{k})}=a_{k+1},\\
a_{k}c_{k}=\nt^{2\nu_{k}(-1)^{k}+1}=\nt^{(-1)^{k}\nu_{k+1}}=b_{k+1},\\
\frac{b_{k}}{a_{k}}=\frac{\nt^{\nu_{k}(-1)^{k+1}}}{\nt^{\nu_{k}(-1)^{k}}}=\nt^{(-1)^{k+1}(2\nu_{k}+(-1)^{k}-(-1)^{k})}=\nt^{(-1)^{k+1}\nu_{k+1}+1}=c_{k+1},
\end{array}
\]
thus the result.
\end{proof}
\begin{rem}
When $|\nt|\neq1$, the sequence of unassuming arrangements $(\tilde{\cC}_{k})$
tends to $x^{2}y{}^{2}z{}^{2}=0.$ The sequence of the $15$ double
points of $(\tilde{\cC}_{k})$ tends to the set of three points $\{(1:0:0),(0:1:0),(0:0:1)\}$.
\\
Working modulo projective automorphism, these $15$ points tend to
a set $P_{9}$ of $9$ points (given by the columns of the $3\times15$
matrix in Section \eqref{subsec:Images of six lines}, with $\nt$
replaced by $0$). The dual arrangement $\cD(P_{9})$ is the simplicial
line arrangement $A(9,48)$ (see \cite{Ctz} for the notations for
simplicial arrangements); it is fixed by $\L_{\{3\},\{2\}}$.  The
line arrangement $\mathcal{L}_{\{2\}}(P_{9})$ is the simplicial arrangement
$A(13,96)_{3}$.
\end{rem}

\begin{defn}
\label{def:Triangle}Let $\cC$ be an unassuming arrangement. Then
the line arrangement $\cC\cup\cC'$ (where $\cC'=\ldt(\cC)$) has
exactly three $4$-points $p_{1},p_{2},p_{3}$. By induction, these
points are common to all iterates of $\cC$ under $\ldt$; we call
these three points the \textit{base points} of $\cC$, and we call
the lines dual to these points \textit{the triangle of} $\cC$. The
triangle of $\cC$ is contained in the dual line arrangement $\check{\cC}=\cL_{2}(\cD(\cC))$. 
\end{defn}

An unassuming arrangement with base points $\mathcal{T}_{3}=\{(1:0:0),(0:1:0),(0:0:1)\}$
is projectively equivalent under a projective map $g$ to some arrangement
$\tilde{\cC_{0}}=\cC(1,1,t)$, which has the same base points. The
maps $g$ must preserves the $3$ base points and therefore is the
product of a diagonal map $(x:y:z)\to(ux:vy:wz)$ (for some $u,v,w$)
and a permutation of the coordinates (which permutation has no effect
on the unassuming arrangement). Thus we obtain that
\begin{lem}
The unassuming arrangements $\cC(a,b,c)$ in \eqref{eq:abc} are all
unassuming arrangements with base points in $\mathcal{T}_{3}$. 
\end{lem}

\subsubsection{\label{subsec:The-cross-ratio}Cross ratio }

The triangle (see Definition \ref{def:Triangle}) of the unassuming
line arrangements $\cC(1,1,\nt)$ is the union of the lines $L_{1}:x=0$,
$L_{2}:y=0$, $L_{3}:z=0$. We define a parametrization of these lines
by $a\to(0:a:1),$ $b\to(b:0:1)$, $c\to(c:1:0)$. We remark that
if one defines sequences $A_{k},B_{k},C_{k}$ of pairs of points in
$L_{1},L_{2},L_{3}$ by respectively $A_{k}=\{a_{k},-a_{k}\}$, $B_{k}=\{b_{k},-b_{k}\}$
and $C_{k}=\{c_{k},-c_{k}\}$, where we recall that 
\[
a_{k}=t^{\nu_{k}(-1)^{k}},\,\,b_{k}=t^{\nu_{k}(-1)^{k+1}}=\frac{1}{a_{k}},\,\,c_{k}=t^{\nu_{k}(-1)^{k}+1}=ta_{k},
\]
we obtain the relations 

\begin{equation}
\begin{array}{c}
\cros(A_{k-1};A_{k})=\cros(B_{k-1};B_{k})=\cros(C_{k-1};C_{k})=\{\left(\frac{t^{2^{k}}-1}{t^{2^{k}}+1}\right)^{\pm2}\}\end{array}.\label{eq:Cross-1}
\end{equation}
the points in $A_{k},B_{k},C_{k}$ seen as elements in $L_{1},L_{2},L_{3}$
defines the normals of the lines of $\tilde{\cC}_{k}$. 

\section{\label{sec:dynamics-and-Applications}dynamics and Applications }

\subsection{\label{subsec:On-the-dynamic}On the dynamics of $\chi(z)=-\tfrac{1}{z^{2}}$
and $F_{\protect\L}=2z^{2}-1$}

For basic notions on dynamic we refer to the books \cite{Beardon}
and \cite{Silverman}.

\subsubsection{\label{subsec:The-semi-conjugacy-between}The semi-conjugacy between
$\chi$ and $F_{\protect\L}$}

Let us also recall that the rational function $F_{\L}:z\to2z^{2}-1$
gives the action of $\L_{\{2\},\{3\}}$ on the moduli space $\kBS/S_{6}=\PP^{1}$
of unassuming arrangements. We obtained $F_{\L}$ by making the following
diagram to commute 
\[
\begin{array}{ccc}
\cC_{0}(t)\in\kBS & \stackrel{}{\to} & \text{P}(\cC_{1}(t))\in\kBS\\
\downarrow &  & \downarrow\\
\kBS/S_{6} & \stackrel{F_{\L}}{\to} & \kBS/S_{6}
\end{array}
\]
Here $\text{P}(\cC_{1}(\nt))$ the image of the line arrangement $\cC_{1}(\nt)=\L_{\{2\},\{3\}}(\cC_{0}(\nt))$
by the projective transformation sending the first four lines to the
canonical basis, so that it is in $\kBS$. The horizontal map is
given in the period map (see \eqref{eq:gamma-1}) $\varUpsilon:\nt\to\tfrac{1}{2}(\nt^{2}+\nt^{-2})$
and the above diagram is 
\[
\begin{array}{ccc}
\PP^{1} & \stackrel{\chi}{\to} & \PP^{1}\\
\varUpsilon\downarrow &  & \varUpsilon\downarrow\\
\PP^{1} & \stackrel{F_{\L}}{\to} & \PP^{1}.
\end{array}
\]
It is commutative so that one has
\begin{equation}
F_{\L}\circ\varUpsilon(\nt)=\varUpsilon\circ\chi(\nt),\label{eq:F_L-and-chi}
\end{equation}
in other words: the map $\Upsilon$ is a semi-conjugacy between $\chi$
and $F_{\L}$. A consequence is that if $\nt$ is a periodic point
of $\chi$, then it is a periodic point of $F_{\L}$ (we give below
an example for which the period of $F_{\L}(\nt)$ is $2$ for $F_{\L}$,
but the period of $\nt$ for $\chi$ is $4$).  

\subsubsection{\label{subsec:Julia-sets-of}Julia sets of $\chi:z\to-\tfrac{1}{z^{2}}$
and $F_{\protect\L}:z\to2z^{2}-1$}

{} The orbit of $0$ under $\chi$ is $\{0,\infty\}$. Consider $z=\rho e^{i\t}\in\CC$,
$z\neq0$ with $\rho\in\RR^{+}$, $\t\in\RR$. If $\rho\neq1$, the
sequence $\chi^{2n}(z)$ (resp. $\chi^{2n+}(z)$) converges to $0$
(resp. $\infty$) if $\rho<1$ and to $\infty$ (resp. $0$) if $\rho>1$.
The Julia set of $\chi$ is the unit circle. 

The polynomial $F_{\L}$ is the degree two Tchebychev polynomial,
one has: 
\[
F_{\L}(\cos(\t))=\cos(2\t).
\]
 The Julia set of $F_{\L}$ is the interval $[-1,1]$ (see e.g., \cite{Beardon}). 

The map $\chi$ restricted to the unit ball (identified with $\RR/2\pi\ZZ$)
is the map $\theta\to-2\t\in\RR/2\pi\ZZ$. The map $F_{\L}$ restricted
to $[-1,1]$ maps $\cos(\t)$ to $\cos^{2}(\t)-1=\cos(2\t)$. The
quotient map $\varUpsilon:z\to\tfrac{1}{2}(\nt^{2}+\nt^{-2})$ restricted
to the unit ball is the map 
\[
\theta\in\RR/2\pi\ZZ\to\cos(2\t).
\]
The image of the Julia set of $\chi$ by $\varUpsilon$ is the Julia
set of $F_{\L}$.

\subsubsection{\label{subsec:Periodic-and-pre-periodic}Periodic and pre-periodic
sequences $(\protect\cC_{k})_{k\in\protect\NN}$.}

Let us study when the sequence $\tilde{\cC_{k}}$ with first term
$\tilde{\cC}_{0}(\nt)=\cC(1,1,\nt)$ is periodic. We recall that one
has $\tilde{\cC_{k}}=\cC(a_{k},b_{k},c_{k})$ where
\[
a_{k}=\nt^{\nu_{k}(-1)^{k}},\,\,b_{k}=\nt^{\nu_{k}(-1)^{k+1}}=\frac{1}{a_{k}},\,\,c_{k}=\nt^{\nu_{k}(-1)^{k}+1}=\nt a_{k},
\]
and $\nu_{k}=\tfrac{1}{3}(2^{k}+(-1)^{k+1}).$ The sequence $(\tilde{\cC_{k}})_{k}$
is therefore periodic if and only if $(a_{k})_{k}$ is periodic. If
the sequence $(a_{k})_{k}$ is periodic with period $n$ then
\[
\nt^{\nu_{n-1}(-1)^{n-1}}=1,
\]
therefore $\nt$ is an $\nu_{n-1}$ root of unity. Conversely, if
$\nt$ is an $m$-root of unity, with $m$ dividing $\nu_{n-1}$ and
not dividing $\nu_{k}$ for $k\leq n-2$, we obtain an $n$-periodic
sequence of unassuming arrangements. 

\subsubsection{\label{subsec:The-ope-invertible}Invertibility of the operator $\protect\L_{\{2\},\{3\}}$}

Using the formula in \eqref{eq:abc}, one may search for the antecedents
of the configuration $\cC(u,v,w)$.
\begin{thm}
The operator $\L_{\{2\},\{3\}}$ has an inverse on the space of real
unassuming arrangements, and has degree two on the space of complex
unassuming arrangements. 
\end{thm}

\begin{proof}
An arrangement $\cC_{-1}(a,b,c)$ is such that $\L_{\{2\},\{3\}}(\cC_{-1})=\cC(u,v,w)$
if and only if the parameters $a,b,c$ are such that 
\[
\begin{array}{ccc}
\frac{b}{a}=\pm u, & ac=\pm v, & \frac{b}{c}=\pm w.\end{array}
\]
Over the complex field, the solutions $(a,b,c)$ are among the triplets
\begin{equation}
(\pm\rho,\pm w\rho,\pm\frac{w}{u}\rho),\label{eq:rho}
\end{equation}
where $\rho=\sqrt{\pm\frac{uv}{w}}$. The choice of the three signs
in Equation \eqref{eq:rho} does not matter for the line arrangement
$\cC_{-1}$, since that choice only permutes the marking of the arrangement.
Thus we can suppose that $(a,b,c)$ is equal to $(\rho,w\rho,\frac{w}{u}\rho)$,
and there are two possibilities for $\rho$: either $\rho_{1}=\sqrt{\frac{uv}{w}}$
or $\rho_{2}=\sqrt{-\frac{uv}{w}}$. Over the complex numbers, one
has two antecedents. Suppose that the unassuming arrangement $\cC(u,v,w)$
is real i.e., $u,v,w\in\RR$. Then, only one of the two numbers $\rho_{1},\rho_{2}$
is real, and therefore there is a unique antecedent for a real unassuming
arrangement. %
\end{proof}

\subsection{\label{subsec:Sing-of-the-iterate}Singularities of the union of
the iterates of $\protect\ldt$}

Let $(\cC_{k})_{k\geq0}$ be the sequence of line arrangements inductively
defined by $\cC_{0}=\cC(1,1,\nt)$ for $\nt\in\PP^{1}$ and $\cC_{k+1}=\ldt(\cC_{k})$. 
\begin{prop}
For $\nt$ in a dense open set, the line arrangement $\cA_{k}=\cup_{0\leq j\leq k}\cC_{j}$
has singularities 
\[
t_{2}(\cA_{k})=12(k^{2}-k+1),\,\,t_{3}(\cA_{k})=12k,\,\,t_{2k+2}(\cA_{k})=3,
\]
and $t_{r}=0$ for $r\notin\{2,3,2k+2\}$. 
\end{prop}

\begin{proof}
For $\nt$ in a dense open set of $\PP^{1}$, the sequence $(\nt^{\nu_{k}})_{k\geq2}$
has only distinct elements, and the line arrangement $\cA_{k}$ is
the union of $6(k+1)$ distinct lines (see Theorem \ref{thm:Formula-ABC}).
We proceed by induction. All the lines in $\cA_{k}$ pass through
one of the three base points $p,q,r$ (see Definition \ref{def:Triangle}).
That implies that apart in these three points, the singularities of
$\cA_{k}$ can only be double or triple points. The line arrangement
$\cC_{k}\cup\cC_{k+1}$ has singularities 
\[
t_{2}=t_{3}=12,\,\,t_{4}=3.
\]
The $4$-points of $\cC_{k}\cup\cC_{k+1}$ are at $p,q,r$, these
points are the three $2k+2$-points of $\cA_{k}$, thus $\cA_{k+1}$
has three $2k+4$-points. Any line of $\cC_{k+1}$ passes through
three of the $15$ double points of $\cC_{k}$. The triple points
of $\cC_{k}\cup\cC_{k+1}$ are obtained as intersection points of
two lines of $\cC_{k}$ and one line in $\cC_{k+1}$. The $12$ nodes
of $\cC_{k}\cup\cC_{k+1}$ are the $12$ nodes of $\cC_{k+1}$ distinct
from $p,q,r$. The configuration $\cC_{k+1}$ meets the $6k$ lines
of $\cA_{k-1}$ in $6\cdot6k-6\cdot2k=24k$ double points. Thus the
number of double points of $\cA_{k+1}$ is 
\[
12(k^{2}-k+1)+24k=12(k^{2}+k+1)=12((k+1)^{2}-(k+1)+1),
\]
and the number of triple points of $\cA_{k+1}$ is $12(k+1)$. 
\end{proof}

\subsection{\label{subsec:Periodic-points}Periodic points of $\protect\L_{\{2\},\{3\}}$,
$\chi$ and $F_{\protect\L}$}

In Table \ref{tab:First-periodic-configurations}, we study the
line arrangement $\cC_{0}(\nt)$ and its orbit by $\L_{\{2\},\{3\}}$
for various roots of unity $\nt$, so that the operator $\L_{\{2\},\{3\}}$
is periodic. In the first column is the value of $\nt$, the columns
$2,3,4,$ are the period under the action of $\L_{\{2\},\{3\}}$,
$\chi$ and $F_{\L}$ respectively. The last column gives in the first
line the invariants of the arrangement $\cC=\cup_{k\geq0}\cC_{k}(\nt)$,
and when possible or needed, in the second line it gives invariants
or identification of the line arrangement $\overline{\cC}$ which
is obtained by taking the union of the orbits $\cO(\cC_{k}(\nt'))$
when $\nt'$ varies among the conjugates of $\nt$. We denote by $\ze_{k}$
a primitive $k$-root of unity.

\begin{table}
\begin{tabular}{|c|c|c|c|c|}
\hline 
Point $\nt$ & {\footnotesize{}$\cC_{0}$ fixed by} & {\footnotesize{}$\nt$ fixed by} & {\footnotesize{}$\varUpsilon\nt$ fixed by} & Configuration\tabularnewline
\hline 
$\ze_{3}$ & $\L_{\{2\},\{3\}}^{3}$ & $\chi$ & $F_{\L}$ & $\text{Ceva}(6)$\tabularnewline
\hline 
$\ze_{5}$ & $\L_{\{2\},\{3\}}^{4}$ & $\chi^{4}$ & $F_{\L}^{2}$ & {\footnotesize{}$\begin{array}{c}
t_{2}=t_{3}=48,t_{8}=3\\
\overline{\cC}=\text{Ceva}(10)
\end{array}$}\tabularnewline
\hline 
$\ze_{11}$ & $\L_{\{2\},\{3\}}^{5}$ & $\chi^{5}$ & $F_{\L}^{5}$ & {\footnotesize{}$\begin{array}{c}
t_{2}=2t_{3}=120,t_{10}=3\\
\overline{\cC}=\text{Ceva}(22)
\end{array}$}\tabularnewline
\hline 
$\ze_{7}$ & $\L_{\{2\},\{3\}}^{6}$ & $\chi^{6}$ & $F_{\L}^{3}$ & {\footnotesize{}$\begin{array}{c}
t_{2}=72,t_{3}=120,t_{12}=3\\
\overline{\cC}=\text{Ceva}(14)
\end{array}$}\tabularnewline
\hline 
$\zeta_{21}$ & $\L_{\{2\},\{3\}}^{6}$ & $\chi^{6}$ & $F_{\L}^{6}$ & {\footnotesize{}$\begin{array}{c}
t_{2}=216,t_{3}=72,t_{12}=3\\
\overline{\cC}=\text{Ceva}(42)
\end{array}$}\tabularnewline
\hline 
$\zeta_{43}$ & $\L_{\{2\},\{3\}}^{7}$ & $\chi^{7}$ & $F_{\L}^{7}$ & {\footnotesize{}$\begin{array}{c}
t_{2}=336,\,t_{3}=84,t_{14}=3\\
\overline{\cC}=\text{Ceva}(86)
\end{array}$}\tabularnewline
\hline 
$\zeta_{17}$ & $\L_{\{2\},\{3\}}^{8}$ & $\chi^{8}$ & $F_{\L}^{4}$ & {\footnotesize{}$\begin{array}{c}
t_{2}=480,t_{3}=96,t_{16}=3\\
\overline{\cC}=\text{Ceva}(34)
\end{array}$}\tabularnewline
\hline 
$\zeta_{9}$ & $\L_{\{2\},\{3\}}^{9}$ & $\chi^{3}$ & $F_{\L}^{3}$ & $\text{Ceva}(18)$\tabularnewline
\hline 
$\zeta_{19}$ & $\L_{\{2\},\{3\}}^{9}$ & $\chi^{9}$ & $F_{\L}^{9}$ & {\footnotesize{}$\begin{array}{c}
t_{2}=432,t_{3}=180,t_{18}=3\\
\overline{\cC}=\text{Ceva}(38)
\end{array}$}\tabularnewline
\hline 
$\zeta_{31}$ & $\L_{\{2\},\{3\}}^{10}$ & $\chi^{10}$ & $F_{\L}^{5}$ & {\footnotesize{}$\begin{array}{c}
t_{2}=840,\,t_{3}=120,t_{20}=3\\
\overline{\cC}=\text{Ceva}(62)
\end{array}$}\tabularnewline
\hline 
$\zeta_{15}$ & $\L_{\{2\},\{3\}}^{12}$ & $\chi^{4}$ & $F_{\L}^{4}$ & {\footnotesize{}$\begin{array}{c}
t_{2}=t_{3}=432,t_{24}=3\\
\overline{\cC}=\text{Ceva}(30)
\end{array}$}\tabularnewline
\hline 
$\zeta_{13}$ & $\L_{\{2\},\{3\}}^{12}$ & $\chi^{12}$ & $F_{\L}^{6}$ & {\footnotesize{}$\begin{array}{c}
t_{2}=144,t_{3}=528,t_{24}=3\\
\overline{\cC}=\text{Ceva}(26)
\end{array}$}\tabularnewline
\hline 
$\zeta_{33}$ & $\L_{\{2\},\{3\}}^{15}$ & $\chi^{5}$ & $F_{\L}^{5}$ & {\footnotesize{}$\begin{array}{c}
t_{2}=2t_{3}=1080,t_{30}=3\\
\overline{\cC}=\text{Ceva}(66)
\end{array}$}\tabularnewline
\hline 
$\ze_{23}$ & $\L_{\{2\},\{3\}}^{22}$ & $\chi^{22}$ & $F_{\L}^{11}$ & {\footnotesize{}$\begin{array}{c}
t_{2}=264,t_{3}=1848,t_{44}=3\\
\overline{\cC}=\text{Ceva}(46)
\end{array}$}\tabularnewline
\hline 
$\ze_{29}$ & $\L_{\{2\},\{3\}}^{28}$ & $\chi^{28}$ & $F_{\L}^{14}$ & {\footnotesize{}$\begin{array}{c}
t_{2}=336,t_{3}=3024,t_{56}=3\\
\overline{\cC}=\text{Ceva}(58)
\end{array}$}\tabularnewline
\hline 
\end{tabular}\caption{\label{tab:First-periodic-configurations}First periodic configurations}

\end{table}

For the values of $\zeta_{k}$ in Table \ref{tab:First-periodic-configurations},
the union of the periodic line arrangements $\cC_{k}$ is contained
in some $\text{Ceva}(m(k))$ line arrangement, for some $m(k)\in\NN$.
It would be interesting to know if this is true for any $k$. By the
following remarks, necessarily one has $\forall k,\,\,m(k)\notin\{3,5,7,9\}$: 
\begin{rem}
\label{rem:By-doing-a}Let be $m\in\{3,5,7,9\}$. By doing a systematic
search among the sub-arrangements of six lines of $\text{Ceva}(m)$,
one obtains that $\text{Ceva}(m)$ does not contain an unassuming
arrangement. 

Let $N(k)$ be the order of $-2$ in the multiplicative group $(\ZZ/k\ZZ)^{*}$.
Since $\chi$ is the map $\nt\to-\tfrac{1}{\nt^{2}}$, a primitive
$k$-root of unity $\ze_{k}$ with $k$ odd is periodic of period
$N(k)$ for $\chi$.
\end{rem}

\subsection{\label{subsec:Antecedents-of-,}Antecedents of $0$ by $\chi$}

The configuration $\cC_{0}=\cC_{0}(i)$ (for $i^{2}=-1$) is such
that $\varUpsilon(i)=1$. The arrangement $\cC_{1}=\L_{\{2\},\{3\}}(\cC_{0})$
is the complete quadrilateral and $\cC_{2}=\L_{\{2\},\{3\}}(\cC_{1})$
is equal to $\emptyset$, moreover 
\begin{prop}
The arrangement $\cA_{12}=\cC_{0}\cup\cC_{1}$ has $12$ lines with
singularities
\[
t_{3}=16,\,t_{4}=3,\,t_{r}=0\text{ for }r\neq3,4.
\]
The lines and the triple points form a $(16_{3},12_{4})$-configuration,
which is the $\text{Ceva}(4)$ arrangement.
\end{prop}

For $n\leq8$ and $\b\in(\chi^{n})^{-1}(0)$, we computed that the
sequence $(\cC_{k}(\b))_{k}$ has length $n+1$: $\cC_{0}(\b),\dots,\cC_{n-1}(\b)$
are arrangements of $6$ lines in $\kBS$, $\cC_{n}$ is the complete
quadrilateral, and $\cC_{k}=\emptyset$ for $k\geq n$. By taking
the union of arrangements $\cC_{k}(\b)$ with $\b\in(\chi^{n})^{-1}(0)$,
one obtains the configurations $\text{Ceva}(2^{n})$. %

\subsection{\label{subsec:Disassembling-Hesse}Assembling the Hesse configuration}

Let $\cC_{0}$ be the configuration of $6$ lines such that their
normals are the columns of following matrix
\begin{equation}
\left(\begin{array}{cccccc}
1 & 0 & 0 & 1 & -j & \frac{1}{3}(1-j)\\
0 & 1 & 0 & 1 & j+1 & \frac{1}{3}(j+2)\\
0 & 0 & 1 & 1 & 1 & 1
\end{array}\right),\label{eq:HesseC}
\end{equation}
where $j^{2}+j+1=0$. This is an unassuming arrangement: the configuration
$\L_{2}(\cD(\cC_{0}))$ has singularities $t_{2}=27,\,t_{3}=6,\,t_{5}=6$
and the $5$-points are not contained in a conic. A direct computation
shows that:
\begin{prop}
The associated $\L_{\{2\},\{3\}}$-sequence $(\cC_{k})_{k\in\NN}$
is periodic with period $2$ and $\cH=\cC_{0}\cup\cC_{1}$ is the
Hesse configuration of $12$ lines with singularities $t_{2}=12,t_{4}=9$.
\\
The arrangements $\cC_{0}$ and $\cC_{1}$ corresponds to the two
points of the labelled moduli space of realization of the second matroid
in Section \ref{subsec:Matroids-and-six}.
\end{prop}

\begin{rem}
The moduli space in Section \ref{subsec:Matroids-and-six} has two
points, but this is the moduli of labelled line arrangements, forgetting
the marking gives projectively equivalent arrangements.
\end{rem}

One can check moreover that the Hesse configuration contains exactly
$6$ arrangements $\cC$ of six lines having only nodal singularities.
Each of these line arrangement $\cC$ is an unassuming arrangement.
These six arrangements are projectively equivalent: they form an orbit
of an order $6$ projective automorphism of $\PP^{2}$. 

We remark that the dual Hesse arrangement is also known as the $\text{Ceva}(3)$
arrangement. 

Finally, consider $\check{\co}=\cL_{2}(\cD(\co))$ and $\check{\cC_{1}}=\L_{\{3\},\{2\}}(\check{\co})$.
One has 
\[
\L_{\{3\},\{2\}}(\check{\cC_{1}})=\check{\cC_{0}}.
\]
The union $\check{\co}\cup\check{\cC_{1}}$ has $21$ lines and singularities
$t_{2}=36,t_{4}=9,t_{5}=12$; it is the reflexion line arrangement
associated to the complex reflexion group $G_{26}$ in Shephard-Todd
classification. 

\subsection{\label{subsec:The-cubic-surface}Cubic surfaces and unassuming arrangements}

The aim of this section is to prove the following result:
\begin{prop}
\label{prop:AllModuli}Let $\co$ be an unassuming arrangement. Then
$\co$ is in the family $\kBS/S_{6}$ or in the second moduli space
described in Section \ref{subsec:Matroids-and-six}.
\end{prop}

That result on matroids has been checked by L. Kühne (personal communication)
using combinatorial tool. That proves Theorem \ref{thm:Main1} without
assumption on the six dual points not being on a conic. We give below
a geometric proof.

Rather than studying $\cC_{0}$, we will be using its dual $\check{\cC_{0}}=\cL_{2}(\cD(\co))$.
We say that a set $P_{6}$ of six points $p_{1},\dots,p_{6}$ in the
plane is unassuming if these points are dual to an unassuming arrangement,
which means that the points are not contained in a conic and $\check{\cC_{0}}=\cL_{2}(P_{6})$
is a line arrangement of $15$ lines with singularities
\[
t_{2}=27,\,t_{3}=t_{5}=6.
\]
For classifying these arrangements, we will use the classification
of cubic surfaces and their automorphism groups.

Our reference for cubic surfaces is \cite[Section V.4]{Hartshorne}.
Let $P_{6}=\{p_{1},\dots,p_{6}\}$ be an arrangement of points in
the plane such that no three are on a line and not all are contained
in a conic. Let $\pi:X\to\PP^{2}$ be the blow-up at the points in
$P_{6}$. By taking the pull-back to $X$ of the linear system of
cubics with base points $p_{1},\dots,p_{6}$, one obtains a morphism
$X\to\PP^{3}$; it is an embedding and we identify $X$ with its image,
which is a smooth cubic surface. 

The cubic surface $X$ contains exactly $27$ lines, which are the
strict transform $F_{ij}$ of the $15$ lines through points $p_{i},p_{j}$,
($1\leq i<j\leq6$), the six exceptional curves $E_{1},\dots,E_{6}$
above the points $p_{1},\dots,p_{6}$, the strict transform $G_{1},\dots,G_{6}$
of the six conics $C_{1},\dots C_{6}$ containing five of the six
points $p_{1},\dots,p_{6}$ (we denote by $G_{k}$ the strict transform
of the conic which does not contain $p_{k}$). Conversely, if $E_{1}',\dots,E_{6}'$
are six disjoint lines on a smooth cubic surface, there exists a blow-down
map $\pi':X\to\PP^{2}$ contracting the six lines $E_{j}'$ to points
$p_{1}',\dots,p_{6}'$ and such that the image of the $21$ remaining
lines by $\pi'$ are $15$ lines through the points $p_{j}'$ and
six conics as above.

A point on a smooth cubic surface $X$ which is intersection of three
lines is called an Eckardt point. The image of an Eckardt point by
the blow-down map $\pi:X\to\PP^{2}$ can be \\
a) either a triple point of the line arrangement $\cL_{2}(P_{6})$,
or\\
b) a $5$-point $p_{j}$ of the line arrangement $\cL_{2}(P_{6})$,
in that case, one conic $C_{k}$ among the $5$ conics containing
$p_{j}$ is tangent to one of the lines $L_{ij}$ of $\cL_{2}(P_{6})$
at $p_{j}$, so that the exceptional curve $E_{j},$ and the lines
$F_{ij}$ and $G_{k}$ meet at one point. \\
Conversely, a configuration described in a) or b) corresponds to an
Eckardt point on the surface $X$. Eckardt points on a cubic surface
are of special interest, since to each Eckardt point there corresponds
an involution of $X$ fixing the Eckardt point and an hyperplane.
The automorphism groups of cubic surfaces are classified in \cite{DD}.
In characteristic $0$, cubic surfaces with at least six Eckardt points
are classified as follows

\begin{tabular}{|c|c|c|c|}
\hline 
$|\mathcal{E}|$ & $G$ & Strata & $\dim$\tabularnewline
\hline 
$6$ & $S_{4}$ & $4\text{B}$ & $1$\tabularnewline
\hline 
$9$ & $H_{3}(3)\rtimes\ZZ/2\ZZ$ & $3\text{A}$ & $1$\tabularnewline
\hline 
$9$ & $H_{3}(3)\rtimes\ZZ/4\ZZ$ & $12\text{A}$ & $\{pt\}$\tabularnewline
\hline 
$10$ & $S_{5}$ & $5\text{A}$ & Clebsch\tabularnewline
\hline 
$18$ & $(\ZZ/3\ZZ)^{3}\rtimes S_{4}$ & $3\text{C}$ & Fermat\tabularnewline
\hline 
\end{tabular}

Here $H_{3}(3)$ denotes the Heisenberg group. The column $|\mathcal{E}|$
indicates the number of Eckardt points, the next column is for the
automorphism group of the cubic surface; the surfaces are classified
according to their strata (which is irreducible). When the strata
has a unique point and when it exists, we give the name of the corresponding
surface, otherwise we give the dimension of the strata. 

Let us proceed according to the strata. By definition of the strata,
the abstract configuration of lines and Eckardt points remains the
same in each strata, so that it is enough to known that configuration
in one example. One can find the equations of the cubic surfaces according
to the strata in \cite{DD}, or \cite{Hosoh}.

Case $4\text{B}$. The configurations $\check{\co}(\nt)$ obtained
from the family $\kBS$ give all the cubic surfaces in the strata
$4\text{B}$. Using Magma, one compute that such a cubic has equation
\[
\begin{array}{c}
X_{1}X_{2}X_{3}-\tfrac{1}{2}X_{2}^{2}X_{3}+\tfrac{\nt^{2}+1}{\nt^{2}-1}X_{2}X_{3}^{2}-X_{1}^{2}X_{4}+\tfrac{1}{8}(\nt^{2}+3)X_{2}^{2}X_{4}\\
-\tfrac{3\nt^{2}+1}{\nt^{2}-1}X_{2}X_{3}X_{4}+X_{3}^{2}X_{4}+2X_{1}X_{4}^{2}+\tfrac{1}{8}\tfrac{\nt^{4}+6\nt^{2}+9}{\nt^{2}-1}X_{2}X_{4}^{2}-2X_{3}X_{4}^{2}=0
\end{array}.
\]

Cases $3\text{A}$ and $12\text{A}$. Using Magma, one computes the
$27$ lines and $9$ Eckardt points for some surfaces in strata $3\text{A}$
and $12\text{A}$. Then we check that each of the $27$ lines contains
a unique Eckardt point. Therefore the image of these $9$ points through
any blow-down map $\pi':X\to\PP^{2}$ will be the six $5$-points
and $3$ triple points of the associated $15$ line arrangement.

Case $3\text{C}$. Consider the six points from \eqref{eq:HesseC},
obtained from the Hesse configuration. By blowing-up these points,
one obtains the Fermat cubic surface. Each line of the Fermat cubic
contains two Eckardt points, thus blowing down any set of six disjoint
lines gives a $15$ lines arrangement with six triple points. It correspond
to the rigid unassuming arrangements.

Case $5\text{A}$. The Clebsch cubic is obtained by blowing-up the
set $P_{6}=\{c_{1},\dots,c_{6}\}$ of six points which are the vertices
of a regular pentagon and the center of the pentagon. The line arrangement
$\cL_{2}(P_{6})$, which has $15$ lines, is know under the name of
the  Icosahedral line arrangement. It is also the simplicial arrangement
denoted by $\cA(15)_{120}$ in \cite{Ctz} and it has singularities
\[
t_{2}=15,t_{3}=10,t_{5}=6,
\]
thus the $5$-points $c_{1},\dots,c_{6}$ are not an unassuming arrangement.
The Clebsch cubic surface $X$ possesses $10$ Eckardt points, the
inverse image of the ten triple points under the blow-up map. The
$15$ lines $F_{ij}$ which are the strict transform of the lines
through $c_{i},c_{j}$ form a $(10_{3},15_{2})$-configuration. The
$12$ remaining lines $E_{j},G_{k}$ of $X$ do not contain Eckardt
points. 

Let us prove that using the Clebsch cubic $X$, it is not possible
to construct a set of six disjoint lines such that the image of the
lines on $X$ by the map which contracts these six lines gives an
arrangement of $15$ lines with $t_{2}=27,t_{3}=t_{5}=6$. If one
wants to construct such a set of lines, one must take disjoint lines
$F_{st}$ (each containing two Eckardt points), to be mapped to two
$5$-points, thus we must choose two lines $F_{ij}$, $F_{ij'}$ (which
are disjoint) for some $i,j\neq j'$. The lines that are disjoint
to $F_{ij}$ and $F_{ij'}$ are the $10$ lines $F_{ik},E_{k},G_{k}$
for $k\in\{1,\dots,6\}\setminus\{i,j,j'\}$ and the line $F_{j,j'}$.
One can not take another line of the form $F_{rs}$, moreover it is
easy to check that among the remaining $6$ lines, one cannot choose
$4$ lines disjoint from $F_{i,j}$, $F_{i,j'}$. Thus one cannot
obtain a new unassuming arrangement by using the Clebsch cubic. 

That finishes the proof of Proposition \ref{prop:AllModuli}. 

\subsection{Concluding Remarks}

\subsubsection{Special values of the parameter $\protect\nt$}

Let us define $\cS=\{\pm2\pm\sqrt{5}\}$; these are the values $\nt$
such that the line arrangement $\cC_{0}(\nt)$ defined in \ref{eq:Matrix-Mt}
is not an unassuming arrangement: $\cC_{0}(\nt)$ has $15$ nodes
but the line arrangement $\cL_{2}(\cD(\cC_{0}))=\cD(\cP_{2}(\co))$
has $15$ lines and singularities 
\[
t_{2}=15,t_{3}=10,t_{5}=6.
\]
That line arrangement is the simplicial arrangement $\cA(15)_{120}$
(see \cite{CEL}). 

The image $\cC_{1}'$ by $\ldt$ of $\co=\cC_{0}(\nt)$ is a line
arrangement of $10$ lines with $45$ nodes. The line arrangement
$\co\cup\cC_{1}'$ has $16$ lines with singularities $t_{2}=30,\,t_{4}=15$.
We conjecture that the number of lines of the iterates of $\cC_{0}$
under $\ldt$ diverges toward infinity. 

 In the degenerate cases $\nt\in\{0,\pm1,\infty\}$, one has the
relations $\varUpsilon(1)=\varUpsilon(-1)=1,$ $\varUpsilon(0)=\varUpsilon(\infty)=\infty$
and the peculiarity of these cases is reflected by the map $F_{\L}$
since $F_{\L}(0)=-1,$ $F_{\L}(-1)=1$ and the points $1$ and $\infty$
are fixed points of $F_{\L}$. The image by the period map $\varUpsilon$
of $\nt\in\cS$ equals to $161$ and we remark that the peculiar dynamical
behavior of $\cC_{0}(\nt)$ under the map $\ldt$ is unseen by the
map $F_{\L}$. 

Finally, we notice that the line arrangement $\cC_{1}'$ contains
exactly $5$ subsets of $6$ lines which are unassuming arrangements,
all are such that their image by $\varUpsilon$ equals $51841=2\cdot161^{2}-1$.
Each of the $15$ double points of $\cC_{0}$ is the vertex three
base points (see Definition \ref{def:Triangle}) of unassuming arrangements. 

Figure \ref{fig:SimplicialArr} is a picture of $\cA(15)_{120}$;
the six $5$-points of $\check{\cC_{0}}(\nt)$ (for fixed $\nt\in\cS$)
are in black, the ten triple points are in red or blue, the six points
in red are dual to an unassuming arrangement, the $4$ other sets
of six points with that property are obtained by taking the image
by the order $5$ rotation through the center.

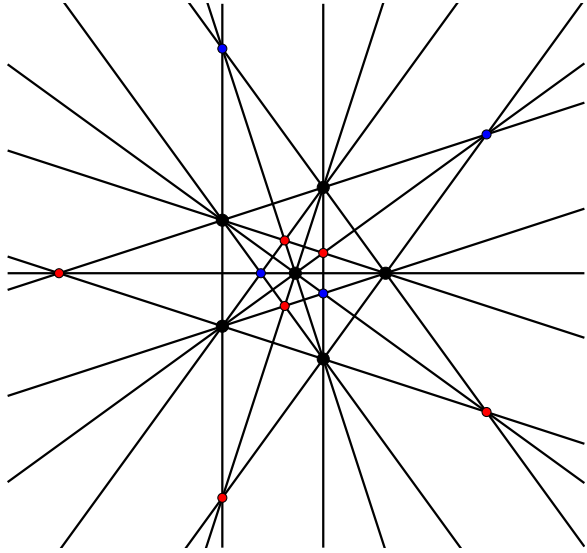
\begin{figure}
\begin{center}

\begin{tikzpicture}[scale=0.4]

\clip(-9.563125000000033,-9.128281249999985) rectangle (9.60687499999997,8.96171874999999);
\draw [line width=0.3mm,domain=-9.563125000000033:9.60687499999997] plot(\x,{(-8.559508646656381--2.8531695488854605*\x)/-2.0729490168751576});
\draw [line width=0.3mm,domain=-9.563125000000033:9.60687499999997] plot(\x,{(-5.2900672706322585--1.7633557568774194*\x)/-5.427050983124842});
\draw [line width=0.3mm,domain=-9.563125000000033:9.60687499999997] plot(\x,{(-0.-0.*\x)/-3.});
\draw [line width=0.3mm,domain=-9.563125000000033:9.60687499999997] plot(\x,{(--5.2900672706322585-1.7633557568774194*\x)/-5.427050983124842});
\draw [line width=0.3mm,domain=-9.563125000000033:9.60687499999997] plot(\x,{(--8.559508646656381-2.8531695488854605*\x)/-2.0729490168751576});
\draw [line width=0.3mm] (0.9270509831248424,-9.128281249999985) -- (0.9270509831248424,8.96171874999999);
\draw [line width=0.3mm,domain=-9.563125000000033:9.60687499999997] plot(\x,{(-0.-2.8531695488854605*\x)/-0.9270509831248424});
\draw [line width=0.3mm,domain=-9.563125000000033:9.60687499999997] plot(\x,{(-5.290067270632258-4.61652530576288*\x)/-3.3541019662496847});
\draw [line width=0.3mm,domain=-9.563125000000033:9.60687499999997] plot(\x,{(-8.559508646656383-1.089813792008041*\x)/-3.3541019662496847});
\draw [line width=0.3mm,domain=-9.563125000000033:9.60687499999997] plot(\x,{(-0.-1.7633557568774194*\x)/2.4270509831248424});
\draw [line width=0.3mm,domain=-9.563125000000033:9.60687499999997] plot(\x,{(-5.290067270632258-4.61652530576288*\x)/3.3541019662496847});
\draw [line width=0.3mm] (-2.4270509831248424,-9.128281249999985) -- (-2.4270509831248424,8.96171874999999);
\draw [line width=0.3mm,domain=-9.563125000000033:9.60687499999997] plot(\x,{(-0.--1.7633557568774194*\x)/2.4270509831248424});
\draw [line width=0.3mm,domain=-9.563125000000033:9.60687499999997] plot(\x,{(-8.559508646656383-1.089813792008041*\x)/3.3541019662496847});
\draw [line width=0.3mm,domain=-9.563125000000033:9.60687499999997] plot(\x,{(-0.--2.8531695488854605*\x)/-0.9270509831248424});
\begin{scriptsize}
\draw [fill=black] (3.,0.) circle (2 mm);
\draw [fill=black] (-2.4270509831248424,1.7633557568774194) circle (2 mm);
\draw [fill=black] (0.9270509831248424,-2.8531695488854605) circle (2 mm);
\draw [fill=black] (0.9270509831248424,2.8531695488854605) circle (2 mm);
\draw [fill=black] (-2.4270509831248424,-1.7633557568774194) circle (2 mm);
\draw [fill=black] (0.,0.) circle (2 mm);
\draw [fill=blue] (-2.4270509831248424,7.469694854648339) circle (1.5 mm);
\draw [fill=red] (-7.8541019662496865,0.) circle (1.5 mm);
\draw [fill=red] (-2.4270509831248424,-7.46969485464834) circle (1.5 mm);
\draw [fill=red] (6.354101966249686,-4.61652530576288) circle (1.5 mm);
\draw [fill=blue] (6.354101966249685,4.61652530576288) circle (1.5 mm);
\draw [fill=red] (-0.35410196624968443,1.0898137920080413) circle (1.5 mm);
\draw [fill=blue] (-1.1458980337503155,0.) circle (1.5 mm);
\draw [fill=red] (-0.3541019662496846,-1.0898137920080413) circle (1.5 mm);
\draw [fill=blue] (0.9270509831248425,-0.6735419648693781) circle (1.5 mm);
\draw [fill=red] (0.9270509831248425,0.6735419648693781) circle (1.5 mm);

\end{scriptsize}

\end{tikzpicture}

\end{center} 

\caption{\label{fig:SimplicialArr}The simplicial arrangement $\protect\cA(15)_{120}$}
\end{figure}

\subsubsection{\label{subsec:Free-arrangements}Free arrangements}

In \cite{OSO}, we study other arrangements of six lines $\cF_{0}$,
which we call flashing arrangements. They have the property that the
image $\cF_{1}=\ldt(\cF_{0})$ is also a flashing arrangement and
$\ldt(\cF_{1})=\cF_{0}$. The arrangement $\cL_{2}(\cD(\cF_{0}))$
has $13$ lines. 

Lukas Kühne pointed out to us the paper \cite{ACKN}, where is studied,
in view of Terao's Conjecture, two line arrangements $L_{13}$ and
$L_{15}$ of $13$ lines and $15$ lines, with one dimensional moduli.
Both arrangements have the property that they are free but non recursively
free arrangements. It turns out that the arrangements of $15$ lines
$\cL_{2}(\cD(\cC))$ associated to the unassuming arrangements $\cC$
in $\kBS$ are the arrangements $L_{15}$. The arrangements $\cL_{2}(\cD(\cF))$
associated to the flashing arrangements are the arrangements $L_{13}$. 

We found our first example of unassuming arrangement (and therefore
the nonbases $\text{NB}_{1}$ (see Section \eqref{subsec:Matroids-and-six})
as a sub-arrangement of $\ldt(\cA(15)_{120})$; the flashing arrangements
were also derived from $\cA(15)_{120}$. We were then searching by
random for a line arrangement $\co$ with the smallest number of lines
possible such that the number of lines of the sequence defined by
$\cC_{k+1}=\ldt(\cC_{k})$ diverges to infinity. 

We found the second nonbases $\text{NB}_{2}$ (see Equation \eqref{eq:triple2})
with a rigid realization space by searching unassuming line arrangements
contained in the Ceva line arrangements.

\vspace{3mm}

\noindent Xavier Roulleau\\
Université d'Angers, \\
CNRS, LAREMA, SFR MATHSTIC, \\
F-49000 Angers, France 

\noindent xavier.roulleau@univ-angers.fr
\end{document}